\documentclass[twoside, 12pt]{amsart}
\usepackage{latexsym}
\usepackage{amssymb,amsmath,amsopn}
\usepackage[dvips]{graphicx}   
\usepackage{color,epsfig}      

\newtheorem{theorem}{Theorem}
\newtheorem{lemma}{Lemma}
\newtheorem{statement}{Statement}
\newtheorem{defi}{Definition}

\DeclareMathOperator{\conv}{conv}

\DeclareMathOperator{\vol}{vol}

\renewcommand{\S}{\mathcal{S}}
\renewcommand{\Re}{\mathbb R}

\begin{document}

\title{On an extremal problem connected with simplices.}
\author{\'A.G.Horv\'ath}
\date{2013 Marc}

\address{\'A. G.Horv\'ath, Dept. of Geometry, Budapest University of Technology,
Egry J\'ozsef u. 1., Budapest, Hungary, 1111}
\email{ghorvath@math.bme.hu}

\subjclass{52A40, 52A38, 26B15}
\keywords{convex hull, isometry, reflection at a hyperplane, simplex, volume inequality.}

\begin{abstract}
In this note we investigate the behavior of the volume that the convex hull of two congruent and intersecting simplices in Euclidean $n$-space can have.
We prove some useful equalities and inequalities on this volume. For the regular simplex we determine the maximal possible volume for the case when the two simplices are related to each other via reflection at a hyperplane intersecting them.
\end{abstract}

\maketitle

\section{Introduction}

The volume of the convex hull of two convex bodies in the Euclidean $n$-space $\Re^n$ has been in the focus of research since the 1950s.
One of the first results in this area is due to F\'ary and R\'edei \cite{fary-redei},
who proved that if one of the bodies is translated along a line with constant velocity, then the volume of their convex hull is a convex function of time.

The quantity $c(K,L)$ was defined in \cite{gho-langi} as follows:

\begin{defi}
For two convex bodies $K$ and $L$ in $\Re^n$, let
\[
c(K,L)=\max\left\{ \vol (\conv (K'\cup L')) : K' \cong K, L' \cong L \mbox{ and } K'\cap L'\neq\emptyset \right\},
\]
where $\vol$ denotes $n$-dimensional Lebesgue measure. Furthermore, if $\S$ is a set of isometries of $\Re^n$, we set
\[
c(K|\S)= \frac{1}{\vol(K)} \max \left\{ \vol( \conv (K \cup K' )) : K \cap K' \neq \emptyset , K' = \sigma(K) \hbox{ for some } \sigma \in \S \right\} .
\]
\end{defi}

We note that a quantity similar to $c(K,L)$ was defined by Rogers and Shephard \cite{rogers-shephard 2}, for which congruent copies were replaced
by translates. Another related quantity is investigated in \cite{gho}, where the author examines $c(K,K)$ in the case that $K$ is a regular tetrahedron and the two congruent copies have the same centroid. In the case when the examined tetrahedra are in dual position, the vertices of the maximal volume configuration forms the vertices of a cube. The author conjectured that this combinatorial assumption can be omitted so that the maximal volume configuration is in every case the vertex set of a cube (see in \cite{gho}). This position of the simplices can be interpreted also via reflection; the respective arrangement is obtained when the two copies are reflected images at their common centroid.

In this paper we consider only simplices. First we recall a result of Rogers and Shephard (see \cite{rogers-shephard 1}) giving the line of a new (immediate) proof for it (Statement 1, Theorem 2). Then, related to this result, we investigate that problem when the set of isometries $\S$ consists of reflections at certain hyperplanes $H$ intersecting the original simplex $S$. We explicitly write the relative volume of the convex hull in Statement 2 and give upper bounds on it (see the Remark after Statement 2 and Theorem 4, respectively). We also determine the number $c(S,S^H)$ for the regular simplex in Theorem 3.

\section{Simplices of dimension $n$}

In this subsection we examine the problem when $K$ and $K'$ are simplices congruent to each other. The first inequality was proved in \cite{rogers-shephard 1}.

\begin{theorem}[\cite{rogers-shephard 1}]
Let $S$ be a simplex of the Euclidean $n$-space with vertices $0=s_0,s_1,\cdots s_n$. Assume that $x\in \mathbb{E}^n$ is a point and $S\cap x+S\neq \emptyset$. Then
$$
\frac{1}{\mathrm{ vol }(S)}\max\left(\mathrm{ vol }(\mathrm{ conv }(S\cup (x+S)))\right)=n+1,
$$
attained at such cases when $S\cap x+S$ is a vertex both of simplices.
\end{theorem}

As a consequence of the inequalities in their Theorem 1, Theorem 2 and Theorem 3 in \cite{rogers-shephard 2} the authors proved the following statement:

\begin{theorem}[\cite{rogers-shephard 2}]
Let $S$ denote a simplex of the Euclidean $n$-space. Assume that $x\in S$ a point and $S_x$ means the simplex  which is the reflected image of $S$ in $x$. Then
$$
\frac{1}{\mathrm{ Vol }(S)}\max\left(\mathrm{ Vol }(\mathrm{ conv }(S\cup S_x))\right)=2^{n}
$$
attained at such a point $x$ which is a vertex of $S$.
\end{theorem}

This theorem can be proved immediately in another way, too. We recall that $P$ is an \emph{extremal point} of the convex body $K$, if there is no segment with endpoints belonging to $K$ which contains $P$ in its relative interior. The following result is useful.

\begin{lemma}[Main Lemma]\label{lem:mainlemma}
If $K$ and $K'$ give a maximal value for $c_{K,K'}$
then the intersection $K\cap K'$ is an extremal point of each of the bodies.
\end{lemma}

This is the case when $K'$ is a translate of $K$, is a reflected image $K^x$ of $K$ in a point $x\in K$ or is a reflected image $K_H$ at a hyperplane $H$, respectively.
In general, to give a proof we need a statement which is interesting for itself. It was proved first in \cite{fary-redei} and later in \cite{rogers-shephard 2}, finally for convex polyhedra of dimension three in \cite{ahn}.

\begin{statement}\label{stm:convexity}
The real valued function $g$ of the real variable $x$ defined by the fixed vector $t$ and the formula
$$
g(x):=\mathrm{ Vol }(\mathrm{ conv }(K\cup (K'+t(x))), \mbox{ where } t(x):=xt,
$$
is convex.
\end{statement}

The proof in \cite{ahn} is based on the observation that the volume change function (by a translation in the direction of a line) can be calculated and it is an increasing function. Since it is also the derivative of $g$ we get that $g$ is convex. This calculation for the volume change can be done in the general case, too. Consider the shadow boundary of the convex hull $\mathrm{ conv }(K\cup (K'+t)$ with respect to the line of translation $t$. This is an $(n-2)$-dimensional topological manifold separating the boundary of
$\mathrm{ conv }(K\cup (K'+t))$ into two domains, the front and back sides of it, respectively. (The translation $t$ can be considered as a motion, hence the respective concepts of front and back sides can be regarded with respect to the direction of it.) Regarding a hyperplane $H$ orthogonal to $t$ the front side and  back side are graphs of functions over the orthogonal projection $X$ of $\mathrm{ conv }(K\cup (K'+t)$ onto $H$. Thus the volume change in $t$ can be calculated by the formula
$$
g'(t)=\lim\limits_{\varepsilon \rightarrow 0}\int_{X} (f^{t+\varepsilon}(X)-f^t(X))+\int_{X} (b^{t+\varepsilon}(X)-b^t(X)),
$$
where, at the moment $t$, $f^t$ and $b^t$ are the graphs of the front ad back sides, respectively. Since $X$ is independent from $t$ and for fixed $X$ the functions
$$
f^{t+\varepsilon}(x)-f^t(x) \mbox{ and } b^{t+\varepsilon}(x)-b^t(x)
$$
in $t$ are increasing and decreasing, respectively, we get that $g'$ is also increasing in $t$ implying that $g$ is convex. (This last statement is not simple and not obvious but it can be proved on an analogous way as in \cite{ahn}. Without loss of generality we can assume that $K$ is a polytope and we can use such elementary observations on the geometric change of the front and back sides as the authors in \cite{ahn}.)

As a corollary we get that if we have two convex, compact bodies $K$ and $K'$ of the Euclidean space of dimension $n$ and they are moving uniformly on two given straight lines then the volume of their convex hull is a convex function of the time. In fact, if the bodies move on the orbits $K+t(x)$ and  $K'+t'(x)$, respectively, then we have
$$
g(x,x')=\mathrm{Vol(conv }\left(K+t(x),K'+t'(x)\right))=\mathrm{Vol(conv }\left(K,K'+t'(x)-t(x)\right))
$$
showing that
$$
g\left(\frac{x_1+x_2}{2},\frac{x'_1-x'_2}{2}\right)=\mathrm{Vol}\left( \mathrm {conv }\left(K,K'+\frac{t'_1+t'_2}{2}-\frac{t_1+t_2}{2}\right)\right)=
$$
$$
=\mathrm{Vol}\left( \mathrm{conv }\left(K,K'+\frac{(t'_1-t_1)+(t'_2-t_2)}{2}\right)\right)\leq
$$
$$
\leq\frac{1}{2}\left(\mathrm{Vol}\left( \mathrm{conv }\left(K,K'+(t'_1-t_1)\right)\right)+ \mathrm{Vol}\left(\mathrm{ conv }\left(K,K'+(t'_2-t_2)\right)\right)\right)=
$$
$$
=\frac{1}{2}\left(g(x_1,x'_1)+g(x_2,x'_2)\right).
$$
Since the function $g(x,x')$ is continuous we get that it is also convex as we stated.

\begin{remark} We emphasize that this statement is not true in hyperbolic space: Let $K$ be a segment and $K'$ be a point which goes on a line in the pencil of the rays ultraparallel to the line of the segment. Since the area function of the triangle defined by the least convex hull of $K$ and $K'$ is bounded (from below and also from above) it cannot be convex function.
\end{remark}

Now a proof for the main lemma can be obtained as follows.

\begin{proof}[Proof of the Main Lemma]
From the convexity of $g$ it follows that the maximal values of $g$ is attained on the boundary of the admissible domain. From this it immediately follows that in the maximal case the intersection contains only one point, and thus it is an extremal point of one of the bodies (e.g., of $K'$). Moreover, if we have a segment belonging to $K$ whose points would be common points of the intersection, then only its endpoints are possible places of the intersection if our bodies giving maximal volume, because in a relative inner point of this segment in one of the opposite directions the investigated volume is increasing by convexity.
\end{proof}
An immediate proof of Theorem 2 now can be obtained as follows:

\begin{proof}[Proof of Theorem 2] Consider the facet $F_0:=\mathrm{ conv }\{S_1,\cdots S_n\}$ of $S$ and the $(n+1)^{th}$ vertex $S_0$. Then we have two possibilities for the convex hull of the union $S\cup S_x$. Denote the reflected images of $S_0$ and $F_0$ at $x$ by $S_{0}^x$ and $F_0^x$, respectively. If  $S_{0}^x$ is also in $S$ then the examined convex hull is the convex hull of the opposite $(n-1)$-dimensional simplices $F_0$ and $F_0^x$, and if $S_{0}^x$ is separated from $S_0$ by the hyperplane $\mathrm{ aff }F_0$, then it is the union of the disjoint parts $\mathrm{ conv }(F_0\cup F_0^x)$, $\mathrm{ conv }(F_0\cup S_0^x)$ and
$\mathrm{ conv }(S_0\cup F_0^x)$. In the first case, the volume is
$$
\mathrm{ Vol }(\mathrm{ conv }(F_0\cup F_0^x))=\frac{2^{n-1}}{n}\mathrm{ Vol }_{n-1}(F_0)d_{F_0},
$$
where $\mathrm{ Vol }_{n-1}(F_0)$ and $d_{F_0}$ means the relative volume of $F_0$ and the distance of the two parallel hyperplanes, respectively. We also have that $d_{F_0}\leq 2m_{F_0}$. The above formula is probably known but the author could not find it in the literature. Thus it is proved here. Dissect the body into two congruent parts by a hyperplane $\mathrm{ aff } \{ S_1,\cdots S_{n-1}, S_1^x, \cdots S_{n-1}^x \}$. Then we get two congruent pyramids based on a body of smaller dimension with analogous properties. If the volume function is $v_n$ we have
$$
v_n=2\frac{1}{n}v_{n-1}a_{n},
$$
where $a_{n}$ is the height of the obtained pyramid corresponding to its base. Using induction we get that
$$
v_n=\frac{2^{n-1}}{n!}v_1a_{n}\cdots a_2=\frac{2^{n-1}}{n!}a_{n}\cdots a_2a_1
$$
where $v_1=a_1$ is the distance of the points $S_1$ and $S_1^x$. The geometric meaning of the product $\frac{1}{n!}a_1\cdots a_{n}$ is the volume of
the simplex $\mathrm{ conv } \{ S_1,\cdots S_{n-1},S_n,S_1^x\} $; thus it is equal to $\frac{1}{n}\mathrm{ Vol }_{n-1}(F_0)d_{F_0}$ showing our formula.

In the second case we have that
$$
\frac{2^{n-1}}{n}\mathrm{ Vol }_{n-1}(F_0)d_{F_0}+\frac{2}{n}\mathrm{ Vol }_{n-1}(F_0)c_{F_0}=\frac{2^{n-1}}{n}\mathrm{ Vol }_{n-1}(F_0)\left(d_{F_0}+ 2^{-(n-2)}c_{F_0}\right),
$$
where  $m_{F_0}=d_{F_0}+c_{F_0}$.  Observe that the possible values of the second function are smaller than the values of the first one. This implies that the maximal value can be attained only in the first case when $2m_{F_0}= d_{F_0}$. In this case $x$ is equal to $S_0$ which is a vertex of $S$, and the volume is equal to
$$
\frac{2^{n-1}}{n}\mathrm{ Vol }_{n-1}(F_0)d_{F_0}=\frac{2^{n-1}}{n}\mathrm{ Vol }_{n-1}(F_0)2m_{F_0}=2^{n}\mathrm{ Vol }(S),
$$
as we stated.
\end{proof}

Before formulating the new results we need some further notation.
Assume that the intersecting simplices $S$ and $S_H$ are reflected copies of each other in the hyperplane $H$. Then $H$ intersects each of them in the same set. By the Main Lemma we have that the intersection of the simplices in an optimal case is a common vertex. Let $s_0\in H$ and $s_i\in H^{+}$ for $i\geq 1$. We imagine that $H$ is horizontal and $H^+$ is the upper half-space. Define the \emph{upper side of $S$} as the collection of those facets in which a ray orthogonal to $H$ and terminated in a far point of $H^+$ is first intersecting $S$.  The volume of the convex hull is the union of those prisms which are based on the orthogonal projection of a facet of the simplex of the upper side. Let denote $F_{i_1},\cdots, F_{i_k}$ the simplex of the upper side, $F'_{i_1},\cdots, F'_{i_k}$ its orthogonal projections on $H$ and $u_{i_1}, \cdots, u_{i_k}$ its respective unit normals, directed outwardly. We also introduce the notation $s=\sum\limits_{i=0}^{n}s_i=\sum\limits_{i=1}^{n}s_i$. Now we have

\begin{statement}
$$
\frac{1}{\mathrm{ Vol }_{n}(S)}\mathrm{ Vol }(\mathrm{ conv }(S,S^H))=2n\sum\limits_{l=1}^k\frac{\langle u_{i_l}, u\rangle \langle u, s-s_{i_l}\rangle}{|\langle u_{i_l}, (n+1)s_{i_l}-s\rangle |}.
$$
\end{statement}

\begin{proof}
The volume of the convex hull is
$$
\mathrm{ Vol }(\mathrm{ conv }(S,S^H)=2\sum\limits_{l=1}^k \mathrm{ Vol }(\mathrm{ conv }(F_{i_l},F'_{i_l}))=2\sum\limits_{l=1}^k \mathrm{ Vol }_{n-1}(F'_{i_l})m_{i_l},
$$
where $m_{i_l}$ is the length of the segment from the centroid of $F_{i_l}$ to the centroid of $F'_{i_l}$. Let $F_{i_l}$ be the face spanned by the vectors $s_i\in S$, where $i\neq i_l$, and assume that $s_{i_1}=s_0=0$. Then
$$
m_{i_l}=\left\{
\begin{array}{ccc}
\frac{1}{n}\sum\limits_{i=1}^n\langle u,s_i\rangle & \mbox{ if } & i_l=0 \\
\frac{1}{n}\sum\limits_{i=1, i\neq i_l}^n\langle u,s_i\rangle & \mbox{ if } & i_l\neq 0
\end{array}
\right.
$$

We now have in the case $0\not \in \{i_1,\cdots ,i_k\}$
$$
\mathrm{ Vol }(\mathrm{ conv }(S,S^H))=2\sum\limits_{l=1}^k\frac{1}{n}\sum\limits_{i\neq i_l}\left\langle u,s_i\right\rangle \cdot \mathrm{ Vol }_{n-1}(F'_{i_l})=\frac{2}{n}\sum\limits_{l=1}^k\sum\limits_{i\neq i_l}\left\langle u\mathrm{ Vol }_{n-1}\left(F'_{i_l}\right),s_i\right\rangle=
$$
$$
=\frac{2}{n}\sum\limits_{l=1}^k\left\langle u\mathrm{ Vol }_{n-1}(F'_{i_l}),\sum\limits_{i\neq i_l}s_i\right\rangle=\frac{2}{n}\sum\limits_{l=1}^k\left\langle \mathrm{ Vol }_{n-1}(F'_{i_l})u, s-s_{i_l}\right\rangle ,
$$
and
$$
\mathrm{ Vol }(\mathrm{ conv }(S,S^H))=\frac{2}{n}\left\langle \mathrm{ Vol }_{n-1}(F'_0)u, s\right\rangle
+\frac{2}{n}\sum\limits_{l=2}^k\left\langle \mathrm{ Vol }_{n-1}(F'_{i_l})u, s-s_{i_l}\right\rangle
$$
in the other case. The two formulas can be written in the following common form:
$$
\mathrm{ Vol }(\mathrm{ conv }(S,S^H))=\frac{2}{n}\sum\limits_{l=1}^k\langle \mathrm{ Vol }_{n-1}(F'_{i_l})u, s-s_{i_l}\rangle
$$
in which $i_l$ also can be zero.

For all $s_{i}$ we have the inequality $\langle u,s_i\rangle\geq 0$ and we know that the equalities
$$
\mathrm{ Vol }_{n}(S)=\frac{1}{n} \left|\left\langle u_{i_l}, s_{i_l}-\frac{1}{n}(s-s_{i_l})\right\rangle \right| \mathrm{ Vol }_{n-1}(F_{i_l})=\frac{1}{n^2}\left|\left\langle u_{i_l},(n+1)s_{i_l}-s\right\rangle\right|\mathrm{ Vol }_{n-1}(F_{i_l})
$$
hold, where $u_{i_l}$ is the unit normal vector of the hyperplane of $F_{i_l}$. On the other hand we have a connection between $\mathrm{ Vol }_{n-1}(F'_{i_l})$ and $\mathrm{ Vol }_{n-1}(F_{i_l})$ of the form
$$
\mathrm{ Vol }_{n-1}(F'_{i_l})=\langle u_{i_l}, u\rangle \mathrm{ Vol }_{n-1}(F_{i_l})
$$
showing that
$$
\mathrm{ Vol }_{n-1}(F'_{i_l})=\frac{\langle u_{i_l}, u\rangle }{|\langle u_{i_l}, (n+1)s_{i_l}-s\rangle |} n^2 \mathrm{ Vol }_{n}(S).
$$
From the above two formulas we have
$$
\mathrm{ Vol }(\mathrm{ conv }(S,S^H))=2n\sum\limits_{l=1}^k\frac{\langle u_{i_l}, u\rangle \langle u, s-s_{i_l}\rangle}{|\langle u_{i_l}, (n+1)s_{i_l}-s\rangle |}\mathrm{ Vol }_{n}(S)
$$
as we stated.
\end{proof}

\begin{remark} The denominator of the $l^{th}$ term of the formula of the statement has a geometric meaning; it is equal to $n$-times the height $m_{i_l}$ of the simplex corresponding to the vertex $s_{i_l}$. In fact, we have
$$
|\langle u_{i_l}, (n+1)s_{i_l}-s\rangle |=(n+1)\left|\langle u_{i_l}, s_{i_l}-\frac{1}{n+1}s\rangle \right|=(n+1)\frac{n}{n+1}|\langle u_{i_l}, s''_{i_l}\rangle |,
$$
where $s''_{i_l}$ is the vector from $s_{i_l}$ to the centroid of the corresponding facet with normal vector $u_{i_l}$. The geometric definition of inner product proves this observation. Thus we have an upper bound on the relative volume:
$$
\frac{1}{\mathrm{ Vol }_{n}(S)}\mathrm{ Vol }(\mathrm{ conv }(S,S^H))\leq 2\sum\limits_{l=1}^k\sum \limits_{i\neq i_l}\frac{\| s_i\|}{m_{i_l}}\leq 2k(n-1)\max\left\{\frac{\|s_i\|}{m_j} \quad i\neq j\right\}.
$$
Observe that this bound is not sharp. However for those simplices which have a small height the relative volume can be large.
\end{remark}

\section{Regular simplices}

The last thought of the previous section motivates the investigation of such simplices for which the ratios $\left\{\frac{\|s_i\|}{m_j} \quad i\neq j\right\}$ are not too large. For example we can solve the original problem in the case of the regular simplex. Denote the Euclidean norm of a vector $x$ by $\|x\|$.

\begin{theorem}
If $S$ is  the regular simplex of dimension $n$, then
$$
c(S,S^H):=\frac{1}{\mathrm{ Vol }_{n}(S)}\mathrm{ Vol }(\mathrm{ conv }(S,S^H))=2n,
$$
attained only in the case when $u=u_0=\frac{s}{\|s\|}$.
\end{theorem}

\begin{proof}
Observe that in this case
$$
u_{i_l}= \left\{
\begin{array}{lcc}
\frac{(n+1)s_{i_l}-s}{\| (n+1)s_{i_l}-s\|} & \mbox{if} & i_l\neq 0 \\
\frac{ s}{\| s\|} & \mbox{if} & i_l=0,
\end{array}
\right.
$$
and
$$
\| (n+1)s_{i_l}-s\|=|\langle u_{i_l}, (n+1)s_{i_l}-s\rangle|=|(n+1)\langle u_{i_l}, s_{i_l}\rangle-\langle u_{i_l}, s_{i_l}\rangle|=n\sqrt{\frac{n+1}{2n}}\|s_1\|.
$$
Without loss of generality we can assume that $\|s_1\|=\ldots =\|s_n\|=1$. It is easy to see that in the case of the regular simplex $i_1=0$ corresponds to an upper facet and thus
$$
\frac{\mathrm{ Vol }(\mathrm{ conv }(S,S^H))}{\mathrm{ Vol }_{n}(S)}=2n\sum\limits_{l=1}^k\frac{\langle u_{i_l}, u\rangle \langle u, s-s_{i_l}\rangle}{|\langle u_{i_l}, (n+1)s_{i_l}-s\rangle |}=
$$
$$
=2n\left(\langle u_0, u\rangle ^2+ \sum\limits_{l=2}^k\frac{{2}\langle  -(n+1)s_{i_l}+s, u \rangle\langle u, s-s_{i_l}\rangle}{{(n+1)n}}\right)=
$$
$$
=2n\left(\langle u_0, u\rangle ^2+ \sum\limits_{l=2}^k\frac{2\left(-(n+1)\langle s_{i_l},u\rangle+\langle s, u \rangle\right)\left(\langle u, s\rangle-\langle u,s_{i_l}\rangle\right)}{(n+1)n}\right)=
$$
$$
=2n\left(\langle u_0, u\rangle ^2+ \sum\limits_{l=2}^k\frac{2\left((n+1)\langle s_{i_l},u\rangle^2-(n+2)\langle s, u \rangle\langle u, s_{i_l}\rangle+\langle u,s\rangle^2\right)}{(n+1)n}\right)=
$$
$$
=2n\left(\langle u_0, u\rangle ^2+ \frac{2}{(n+1)n}\sum\limits_{l=2}^k\left((n+1)\left\langle s_{i_l},u\right\rangle^2-(n+2)\langle s, u \rangle\left\langle u, s_{i_l}\right\rangle+\langle u,s\rangle^2\right)\right)=
$$
$$
=2n\left(\langle u_0, u\rangle ^2+\sum\limits_{l=2}^k\left(\frac{2}{n}\left\langle s_{i_l},u\right\rangle^2-\frac{2(n+2)}{n(n+1)}\langle s, u \rangle\left\langle  s_{i_l},u\right\rangle+\frac{2}{n(n+1)}\langle u,s\rangle^2\right)\right).
$$
If the only upper facet corresponds to the normal vector $u_0$, then only the first term in the formula of Statement 2 occurs -- meaning that $k=1$ --- and the maximal value of the right hand side is less than or equal to $2n$ with equality in the case mentioned in the statement.

Assume now that $k\geq 2$.
By the regularity of the simplex we have that $s=\sqrt{\frac{(n+1)n}{2}}u_0$, hence we get that
$$
\frac{\mathrm{ Vol }(\mathrm{ conv }(S,S^H))}{\mathrm{ Vol }_{n}(S)}:=2nf\left(\left\langle s_{2,k}^0,u\right\rangle,\langle u_0, u \rangle\right)=
$$
$$
=2n\left(\langle u_0, u\rangle ^2+\sum\limits_{l=2}^k\left(\frac{2}{n}\left\langle s_{i_l},u\right\rangle^2-\sqrt{\frac{2}{n}}\frac{(n+2)}{\sqrt{n+1}}\langle u_0, u \rangle\left\langle  s_{i_l},u\right\rangle+\langle u_0,u\rangle^2\right)\right)\leq
$$
$$
\leq2n\left(\langle u_0, u\rangle ^2+\left(\frac{2}{n}\left\langle \sum\limits_{l=2}^ks_{i_l},u\right\rangle-\sqrt{\frac{2}{n}}\frac{(n+2)}{\sqrt{n+1}}\langle u_0, u \rangle\left\langle  \sum\limits_{l=2}^k s_{i_l},u\right\rangle+(k-1)\langle u_0,u\rangle^2\right)\right)=
$$
$$
=2n\left(\left(\frac{\sqrt{2(k-1)k}}{n}\left\langle s_{2,k}^0,u\right\rangle-\sqrt{\frac{(k-1)k}{n(n+1)}}(n+2)\langle u_0, u \rangle\left\langle  s_{2,k}^0,u\right\rangle+k\langle u_0,u\rangle^2\right)\right)=:
$$
$$
=:2ng\left(\left\langle s_{2,k}^0,u\right\rangle,\langle u_0, u \rangle\right),
$$
where  $s_{2,k}:=\sum\limits_{i\neq i_l}s_{i_l}$ and  $s_{2,k}^0:=\frac{s_{2,k}}{\|s_{2,k}\|}$.
Here we used that $0\leq \left\langle s_{i_l},u\right\rangle \leq 1$, implying that $\left\langle s_{i_l},u\right\rangle^2\leq \left\langle s_{i_l},u\right\rangle$.
First we remark that we have $\frac{1}{n}\leq \langle u_0,u\rangle\leq \sqrt{1-\frac{1}{n^2}}$, since the simplex is in the upper half-space determined by the hyperplane of reflection. Furthermore we can observe that if a vertex $s_{i_l}$ gives an upper facet then
$$
\left\langle \sum\limits_{i\neq i_l}\left(s_i-s_{i_l}\right), u\right\rangle\geq 0,
$$
implying that
$$
\left\langle s-(n+1)s_{i_l}, u\right\rangle\geq 0.
$$
From this we get a new connection between the parameters $\left\langle s_{i_l},u\right\rangle$ and $\langle u_0,u\rangle$. We get that
$$
\left\langle s_{i_l},u\right\rangle\leq \frac{\|s\|_2}{(n+1)}\langle u_0,u\rangle=\sqrt{\frac{n}{2(n+1)}}\langle u_0,u\rangle.
$$
This implies that
$$
\left\langle s_{2,k}^0,u\right\rangle\leq (k-1)\sqrt{\frac{2}{(k-1)k}}\sqrt{\frac{n}{2(n+1)}}\langle u_0,u\rangle=\sqrt{\frac{(k-1)n}{k(n+1)}}\langle u_0,u\rangle.
$$
On the other hand, if we write that
$$
\langle u_0,u\rangle:=\cos \alpha  \mbox{, } \langle s_{2,k}^0,u_0\rangle :=\cos \beta \mbox{ and } \langle s_{2,k}^0,u\rangle:=\cos \gamma,
$$
then we get that $\gamma \leq \alpha + \beta$, and so $\cos \alpha \cos \beta-\sin\alpha \sin\beta \leq \cos \gamma$.
But
$$
\cos\beta =\frac{(k-1)+\frac{1}{2}(k-1)(n-1)}{\sqrt{\frac{(k-1)kn(n+1)}{4}}}=\sqrt{1-\frac{n-k+1}{nk}} \mbox{ and } \sin\beta=\sqrt{\frac{n-k+1}{nk}},
$$
hence we have a second inequality which is:
$$
\langle u_0,u\rangle\sqrt{\frac{(n+1)(k-1)}{nk}}-\sqrt{1-\langle u_0,u\rangle^2}\sqrt{\frac{n-k+1}{nk}}\leq  \langle s_{2,k}^0,u\rangle.
$$
Comparing the two inequalities we get a new one on $\langle u_0,u\rangle$. More precisely we have that
$$
\langle u_0,u\rangle\sqrt{\frac{(n+1)(k-1)}{nk}}-\sqrt{1-\langle u_0,u\rangle^2}\sqrt{\frac{n-k+1}{nk}}\leq \sqrt{\frac{(k-1)n}{k(n+1)}}\langle, u_0,u\rangle
$$
implying that
$$
\langle u_0,u\rangle^2\leq \frac{(n+1)(n-k+1)}{(k-1)+(n+1)(n-k+1)}.
$$
Observe that the examined function $g\left(\left\langle s_{2,k}^0,u\right\rangle,\langle u_0, u \rangle\right)$ of two variables is a parable if we fix its first variable $\langle s_{2,k}^0,u\rangle$.  Its maximal value can be found at the boundary of the domain which are at the endpoints $\langle u_0,u\rangle=\frac{1}{n}$ and $\langle u_0,u\rangle=\sqrt{\frac{(n+1)(n-k+1)}{(k-1)+(n+1)(n-k+1)}}$, respectively.

If now we assume that $\langle u_0,u\rangle=\frac{1}{n}$, then we omit the negative (middle) part of the sum and use the inequality on $2nf\left(\left\langle s_{2,k}^0,u\right\rangle,\langle u_0, u \rangle\right)$ to determine an upper bound. We get that it is equal to
$$
\frac{\sqrt{2(k-1)k}}{n}\sqrt{\frac{(k-1)n}{k(n+1)}}\frac{1}{n}+\frac{k}{n^2}=\frac{\sqrt{2\frac{n}{n+1}}(k-1)+k}{n^2}<1,
$$
since $k\leq n-1$ and $n\geq 3$.

In the other case $\langle u_0,u\rangle=\sqrt{\frac{(n+1)(n-k+1)}{(k-1)+(n+1)(n-k+1)}}$, and we get two equalities
$$
\langle s_{2,k}^0,u\rangle=\sqrt{\frac{(k-1)n(n-k+1)}{k\left((k-1)+(n+1)(n-k+1)\right)}}
$$
and
$$
\langle s_{i_l},u\rangle=\sqrt{\frac{n(n-k+1)}{2\left((k-1)+(n+1)(n-k+1)\right)}},
$$
respectively. Using these parameters the value of the original function is
$$
\frac{\mathrm{ Vol }(\mathrm{ conv }(S,S^H))}{\mathrm{ Vol }_{n}(S)}=
$$
$$
=2n\left(\langle u_0, u\rangle ^2+\sum\limits_{l=2}^k\left(\frac{2}{n}\left\langle s_{i_l},u\right\rangle^2-\sqrt{\frac{2}{n}}\frac{(n+2)}{\sqrt{n+1}}\langle u_0, u \rangle\left\langle  s_{i_l},u\right\rangle+\langle u_0,u\rangle^2\right)\right)=
$$
$$
=2n\left(\frac{(n+1)(n-k+1)}{(k-1)+(n+1)(n-k+1)}+\sum\limits_{l=2}^k\left(\frac{n-k+1}{(k-1)+(n+1)(n-k+1)}-\right.\right.
$$
$$
\left.\left.-(n+2)\frac{n-k+1}{(k-1)+(n+1)(n-k+1)} +(n+1)\frac{n-k+1}{(k-1)+(n+1)(n-k+1)}\right)\right)<2n,
$$
showing the truth of the statement.
\end{proof}

\section{Again general simplices}

We note that the result of the case of reflection at a hyperplane gives an intermediate value between the results corresponding to translates and point reflections. The part of the previous proof corresponding to the case of a single upper facet can be extended to a general simplex, too. Let $G$ denote the Gram matrix of the vector system $\{s_1,\ldots,s_n\}$, defined by the product $M^TM$, where $M=[s_1,\cdots,s_n]$ is the matrix with columns $s_i$. In the following theorem we use the notation $\|\cdot\|_1$ for the $l_1$ norm of a vector or a matrix, respectively.

\begin{theorem}
If the only upper facet is $F_0$ with unit normal vector $u_0$, then we have the inequality
$$
\frac{1}{\mathrm{ Vol }_{n}(S)}\mathrm{ Vol }(\mathrm{ conv }(S,S^H))\leq n\left(1+\frac{\|s\|}{\langle u_0,s\rangle }\right)=
$$
$$
=\left(n+\sqrt{\left\|(1,\ldots,1)G^{-1}\right\|_1}\left\|M(1,\ldots,1)\right\|\right).
$$
Equality is attained if and only if the normal vector $u$ of $H$ is equal to $\frac{u_0+s'}{\|u_0+s'\|}$, where $s'=\frac{s}{\|s\|}$ is the unit vector of the direction of $s$.
\end{theorem}

Before the proof we remark that for the regular simplex
$$
G=\left(
   \begin{array}{cccc}
     1 & \frac{1}{2} & \cdots & \frac{1}{2} \\
     \frac{1}{2} & 1 & \cdots & \frac{1}{2} \\
     \vdots & \vdots & \vdots & \vdots \\
     \frac{1}{2} & \cdots & \frac{1}{2} & 1 \\
   \end{array}
 \right) \mbox{ and }
 G^{-1}=\left(
         \begin{array}{cccc}
           \frac{2n}{n+1} & -\frac{2}{n+1} & \cdots & -\frac{2n}{n+1} \\
           -\frac{2}{n+1} & \frac{2n}{n+1} & \cdots & -\frac{2}{n+1} \\
           \vdots & \vdots & \vdots & \vdots \\
           -\frac{2}{n+1} & \cdots & -\frac{2}{n+1} & \frac{2n}{n+1} \\
         \end{array}
       \right),
$$
implying that
$$
n+\sqrt{\left\|(1,\ldots,1)G^{-1}\right\|_1}\left\|M(1,\ldots,1)\right\|=n+\sqrt{\frac{2n}{n+1}}\sqrt{\frac{n(n+1)}{2}}=2n.
$$

\begin{proof}
From our formula we get that in the case when the only upper facet is $F_0$ we have that
$$
\mathrm{ Vol }(\mathrm{ conv }(S,S^H))=2n\sum\limits_{l=1}^k\frac{\langle u_{i_l}, u\rangle \langle u, s-s_{i_l}\rangle}{|\langle u_{i_l}, (n+1)s_{i_l}-s\rangle |}\mathrm{ Vol }_{n}(S)=2n\frac{\langle u_0, u\rangle \langle u, s\rangle}{|\langle u_0,-s\rangle |}\mathrm{ Vol }_{n}(S)=
$$
$$
=2n\frac{\langle u_0, u\rangle \langle u, s'\rangle}{|\langle u_0,-s'\rangle |}\mathrm{ Vol }_{n}(S),
$$
where $s'=\frac{s}{\|s\|}$. Since $\langle u_0,-s'\rangle$ is independent from the choice of $u$ we have to determine the maximal value of
$$
\langle u_0, u\rangle \langle u, s'\rangle.
$$
Using the positivity of these numbers we can see that
$$
\langle u_0, u\rangle \langle u, s'\rangle\leq \left(\frac{\langle u_0, u\rangle+\langle u, s'\rangle}{2}\right)^2=\left(\left\langle \frac{u_0+s'}{2},u\right\rangle\right)^2\leq \left\|\frac{u_0+s'}{2}\right\|^2,
$$
with equality if and only if $\langle u_0, u\rangle =\langle u, s'\rangle$ and $u$ is parallel to $\frac{u_0+s'}{2}$. The second condition implies the first one. Thus we have
$$
u=\frac{u_0+s'}{\|u_0+s'\|},
$$
and hence
$$
\mathrm{ Vol }(\mathrm{ conv }(S,S^H))\leq 2n\frac{\left\|\frac{u_0+s'}{2}\right\|^2}{|\langle u_0,-s'\rangle |}\mathrm{ Vol }_{n}(S)= n\frac{1+\langle u_0,s'\rangle}{\langle u_0,s'\rangle }\mathrm{ Vol }_{n}(S)=n\left(1+\frac{\|s\|}{\langle u_0,s\rangle }\right)\mathrm{ Vol }_{n}(S),
$$
with equality if and only if $$
u=\frac{u_0+s'}{\|u_0+s'\|},
$$
as we stated.

On the other hand, the geometric condition in the statement algebraically means that there are non-negative coefficients $\alpha_i$ such that
$$
u_0=\sum\limits_{i=1}^n \alpha_is_i.
$$
By the definition of $u_0$ for each pair of indices $1\leq i<j\leq n$ the equality $\langle u_0,s_i\rangle=\langle u_0,s_j\rangle$ is also valid. We can explicitly determine the coefficients $\alpha_i$. In fact, we have $1=\sum\limits_{i=1}^n \alpha_i\langle u_0,s_i\rangle$ and so $\frac{1}{\sum\limits_{i=1}^n \alpha_i}=\langle u_0,s_1\rangle=\ldots =\langle u_0,s_n\rangle$. Writing $M=[s_1,\ldots ,s_n]$ we get that $u_0=M(\alpha_1,\ldots,\alpha_n)^T$, and thus
$$
\frac{1}{\sum\limits_{i=1}^n \alpha_i}(1,\ldots,1)=u_0^TM=(\alpha_1,\ldots,\alpha_n)G
$$
where $G$ is the Gram matrix of the vector system $\{s_1,\ldots ,s_n\}$, so $G=[g_{i,j}]=M^TM$. If
$$
(1,\ldots,1)G^{-1}=:(\beta_1,\ldots ,\beta_n)=\sqrt{\sum\limits_{i=1}^n \beta_i}\left(\frac{\beta_1}{\sqrt{\sum\limits_{i=1}^n \beta_i}},\ldots ,\frac{\beta_n}{\sqrt{\sum\limits_{i=1}^n \beta_i}}\right),
$$
then
$$
\alpha_i=\frac{\beta_i}{\sqrt{\sum\limits_{i=1}^n \beta_i}},
$$
since this choice implies
$$
\sum\limits_{i=1}^n \alpha_i=\sqrt{\sum\limits_{i=1}^n \beta_i}.
$$
Thus we have
$$
\left(\alpha_1,\ldots,\alpha_n\right)=\sqrt{\frac{1}{\left\|(1,\ldots,1)G^{-1}\right\|_1}}(1,\ldots,1)G^{-1},
$$
and so
$$
\left\langle u_0,s \right\rangle=\frac{1}{\sqrt{\left\|(1,\ldots,1)G^{-1}\right\|_1}}(1,\ldots,1)G^{-1}G(1,\ldots,1)^T= \frac{n}{\sqrt{\left\|(1,\ldots,1)G^{-1}\right\|_1}},
$$
proving the equality in the statement.
\end{proof}

\section{Acknowledgement}

The author wish to thank Horst Martini for various helpful hints and for the list of concrete errors.

\end{document}